\documentclass[reqno,12pt]{amsart}
\usepackage{amsmath,amssymb,latexsym}	
\usepackage{xcolor,colonequals, comment}
\usepackage{hyperref}
\usepackage{tikz-cd}
\newcommand{\V}{\mathrm{V}}

\newcommand{\M}{\mathrm{M}}
\newcommand{\GL}{\mathrm{GL}}
\newcommand{\SL}{\mathrm{SL}}
\newcommand{\R}{\mathbb{R}}
\newcommand{\Q}{\mathbb{Q}}
\newcommand{\Z}{\mathbb{Z}}
\newcommand{\N}{\mathbb{N}}
\newcommand{\C}{\mathbb{C}}
\def \Aut {\mathrm{Aut}}
\newcommand\numberthis{\addtocounter{equation}{1}\tag{\theequation}}
\def\mcG{\mathcal G}

\newtheorem{theorem}{Theorem}[section]
\newtheorem{proposition}[theorem]{Proposition}
\newtheorem{lemma}[theorem]{Lemma}

\theoremstyle{definition}
\newtheorem{defn}[theorem]{Definition}
\newtheorem{rmk}[theorem]{Remark}

\author{M. M. Radhika}
\address{Kerala School of Mathematics, Kerala, India}
\email{mmr@ksom.res.in}

\author{Sandip Singh}
\address{Department of Mathematics, Indian Institute of Technology Bombay, Mumbai, India}
\email[]{sandip@math.iitb.ac.in}
\begin{document}

\title{Lattices in $\R^n\rtimes\SL_2(\R)$}

\vskip 5mm

\begin{abstract}
We determine the existence of cocompact lattices in groups of the form $\V\rtimes\SL_2(\R)$, where $\V$ is a finite dimensional real representation of $\SL_2(\R)$. It turns out that the answer depends on the parity of $\dim(\V)$ when the representation is irreducible.  
\end{abstract}

\maketitle

\section{Introduction}\label{s1}

\noindent 
A lattice in a Lie group $G$ is a discrete subgroup $\Gamma$ which has a finite covolume with respect to the Haar measure on $G$. Lattices in semisimple Lie groups have been extensively studied. It is known that (linear) semisimple Lie groups contain both cocompact and non-cocompact lattices (cf. \cite{Borel}; \cite{B-H}; \cite{Rag-2}). While non-arithmetic lattices are known to exist in certain orthogonal and unitary groups (see \cite[Appendix C]{GAM}), for instance, in $\mathrm{SO}(1,n)$ (\cite{VSM, EBV} for $n\leq 5$, \cite{GPS}) and $\mathrm{SU}(1,n)$ (\cite{GDM} for $n=2$, \cite{DM} for $n=3$), Margulis arithmeticity theorem asserts that, in most cases, lattices in classical groups are arithmetic subgroups and are obtained by taking the integral points. Arithmeticity allows one to use algebraic and number theoretic techniques to study these lattices, whereas non arithmetic lattices are often studied using geometric techniques. 
 
\noindent
In solvable groups, lattices (if they exist) are cocompact  \cite[Theorem 3.1]{Rag-1}. Both cocompact and non cocompact arithmetic lattices are known to exist in non-compact semisimple groups. 
It is a natural question to ask whether a semidirect product of a solvable group and a semisimple group contains a cocompact lattice or not; whether the semidirect product contains an arithmetic lattice or not; and if such lattices exist, what they are. In this article we will be investigating these questions for the simplest group of this type, $\R^n\rtimes\SL_2(\R)$, where $\R^n$ is an algebraic representation space for $\SL_2(\R)$. 

\noindent
Let $\Gamma$ be a lattice in $\R^n\rtimes\SL_2(\R)$, write $\Gamma_1=\Gamma\cap{\R^n}$ and let $\Gamma_2$ be the projection of $\Gamma$ on $\SL_2(\R)$. It follows from \cite[Corollary 8.28]{Rag-1} (see Appendix A) that $\Gamma_1$ is a lattice in $\R^n$ and $\Gamma_2$ is a lattice in $\SL_2(\R)$.
We show that the lattice $\Gamma_2$ is necessarily an arithmetic lattice in $\SL_2(\R)$ (see Proposition \ref{arithmetic-1}). It is a little surprising that for the irreducible action of $\SL_2(\R)$ the existence of a cocompact lattice in $\R^n\rtimes\SL_2(\R)$ depends on the parity of $n$ (see Theorem \ref{maintheorem}). 
When the action of $\SL_2(\R)$ is irreducible, one obvious way of exhibiting lattices in $\R^n\rtimes\SL_2(\R)$ is to take finite index subgroups of $\Z^n\rtimes\SL_2(\Z)$. When $n$ is even, it turns out that up to conjugates and commensurability these are the only lattices. In particular, for an even $n$, $\R^n\rtimes\SL_2(\R)$ does not contain cocompact lattices. However, when $n$ is odd, $\R^n\rtimes\SL_2(\R)$ does possess cocompact lattices.

\noindent
We now state the main results of this paper. 

\begin{theorem}\label{maintheorem}
For $n\geq{2}$, and the irreducible action of $\SL_2(\R)$ on $\R^n\ (\simeq{\rm{Sym}}^{n-1}(\R^2))$, the following are true:
\begin{itemize}
\item[(1)]  $\R^n\rtimes\SL_2(\R)$ contains a cocompact lattice if and only if $n$ is odd. Moreover, the cocompact lattices are of the form $\Z^n\rtimes\SL_1(\mathcal O_D)$ upto conjugates and commensurability, where $D$ is a quaternion division algebra over a totally real number field $k$.
\item[(2)] If $n$ is even, all lattices are isomorphic to $\Z^n\rtimes\SL_2(\Z)$ up to conjugates and commensurability.
\end{itemize}
\end{theorem}

\noindent
Part (1) of the above theorem is proved in Theorem \ref{ccpt=n odd}.  Although part (1) seems to be a purely Lie theoretic statement, the proof depends on considerations of number theoretic nature. More importantly, the `only if' assertion is achieved using Proposition \ref{ccpt=def|Q} which gives the necessary condition for the existence of a cocompact lattice in terms of the field of definition of the algebraic representation of the associated algebraic group. Part (2) of the above theorem is dealt with in Theorem \ref{lattices in R2.SL(2,R)}. We note here that to exhibit cocompact lattices, we use the classification results for the arithmetic subgroups of $\SL_2(\R)$ which is described, for instance, in \cite[Section 18.5]{DWM}.

\noindent
For a non-trivial and non-irreducible action of  $\SL_2(\R)$ on $\R^n$, we prove the following theorem which is a consequence of Theorem \ref{maintheorem}:

\begin{theorem}
\label{secondmaintheorem}
The group $\R^n\rtimes\SL_2(\R)$ contains a cocompact lattice if and only if the multiplicity of the even dimensional representations occurring in $\R^n$ is even.
\end{theorem}

\noindent
The article is organized as follows. A few relevant definitions and results on arithmetic subgroups of algebraic groups are laid down in Section \ref{notation}. In particular, we discuss arithmetic subgroups in groups of type $A_1$, thereby fixing the notation. Starting from \S\ref{lattice-in-SL2} we specialize to arithmetic subgroups in $\SL_2(\R)$ and thus, fix $k$ to be a totally real number field. Section \ref{s2} discusses when does an algebraic representation defined over $\bar{k}$, of the algebraic $k$-group $\mathbf{SL}_{1,D}$ for $D$ a quaternion division algebra over $k$, descend to a representation defined over $k$. In Section \ref{s3}, we see that the projection $\Gamma_2$ of a lattice $\Gamma$ in $\R^n\rtimes\SL_2(\R)$ on $\SL_2(\R)$, which is again a lattice in $\SL_2(\R)$, is an arithmetic subgroup (see Prop. \ref{arithmetic-1}). In the case when $\Gamma$ is a cocompact, $\Gamma_2$ is cocompact in $\SL_2(\R)$; and the arithmeticity and cocompactness of $\Gamma_2$ forces the action of $\SL_2(\R)$ on $\R^n$ to be defined over $\Q$ (see Prop. \ref{ccpt=def|Q}), or in other words, when $\Gamma$ is cocompact, $\Gamma_2=\SL_1(\mathcal{O}_D)$ for a quaternion division algebra $D$ defined over $\Q$ and split over $\R$ (see Prop. \ref{D_definedover_Q}). Finally, we give the proofs of the main theorems in Section \ref{s4}. Appendix A has comments on \cite[Cor. 8.28]{Rag-1} which is crucially used in this article, and Appendix B states a version of Proposition \ref{direct summand def over k} for abstract groups.

\section{Notation}\label{notation}
\noindent
In this section we set down the notation relevant to the study of arithmetic lattices in semisimple groups, in particular $\SL_2(\R)$, which will be used repeatedly in this article. We begin with stating the notion of arithmetic subgroups of an algebraic group. 
\subsection{Generalities on algebraic groups and arithmetic subgroups}
We let $G$ denote a semisimple algebraic group defined over a number field $k$. The notion of arithmetic subgroup for a linear algebraic groups is as follows. 
\begin{defn}
\label{arith-subgp}
For a linear algebraic group $G$ defined over $\Q$, a subgroup $\Gamma$ of $G$ is said to be an \textit{arithmetic subgroup}, if $\Gamma$ is commensurable with the subgroup $G(\Z)$.
\end{defn}
\noindent
In view of studying the lattices in Lie groups, we will be assuming a more general definition of arithmetic subgroups that respects commensurability, isomorphisms, and allows to ignore compact subgroups. We state here the definition of arithmetic subgroup as given in \cite[Def. 5.1.19]{DWM}.
\begin{defn}\label{arithmeticgroupsdef}
A subgroup $\Gamma$ of a semisimple real Lie group $H$ is an arithmetic subgroup if and only if there exists a connected semisimple algebraic group $G'$ defined over $\Q$, compact normal subgroups $K$ and $K'$ respectively of the identity connected component $H^\circ$ of $H$ and $G'(\R)$, and an isomorphism $\phi:H^\circ/K\rightarrow G'(\R)/K'$ such that $\phi(\overline{\Gamma})$ is commensurable to $\overline{G'(\Z)}$, where $\overline{\Gamma}$ and $\overline{G'(\Z)}$ are respectively the images of $\Gamma$ and $G'(\Z)$ in $H^\circ/K$ and $G'(\R)/K'$.
\end{defn}

\noindent
The cocompactness of arithmetic subgroups in semisimple groups is given by the compactness criterion due to Godement which states that for a semisimple algebraic group $G$ defined over $\Q$, $G(\R)/G(\Z)$ is compact if and only if $G(\Z)$ has no nontrivial unipotent elements.

\noindent
 For the classical groups, a complete classification of the $\Q$-forms \cite[Chapter 2]{P-R}, and hence of their arithmetic subgroups, is available in the literature. We refer the reader to the table in \cite[Section 18.5]{DWM} to know the arithmetic subgroups of a simple Lie group $G(\R)$. In particular, for $G=\mathbf{SL}_{n+1}$ the classification of arithmetic subgroups is achieved using central simple algebras defined over number fields. We collect a few related facts here for completeness.

\noindent
Using Definition \ref{arithmeticgroupsdef} we see that the Weil restriction of scalars functor provides a way to obtain more arithmetic subgroups of simple Lie groups as explained below. For a number field $k$, we let $\Sigma_k^\infty$ denote the set of distinct archimedean valuations of $k$ and $\mathcal{O}_k$ be the ring of integers in $k$. For any faithful embedding of $G$ in $\mathbf{GL}_n(\C)$ for some $n$, we define $G(\mathcal{O}_k)=G\cap\mathbf{GL}_n(\mathcal O_k)$. 
If $\sigma$ is an embedding of $k$ in $\C$, we let $G_\sigma$ denote the algebraic group defined over $\sigma(k)$.
The restriction of scalars to $\Q$ gives the algebraic group 
$\mathrm{Res}_{k/\Q}(G)$ defined over $\Q$ such that $\mathrm{Res}_{k/\Q}(G)(\C)=\prod_{\sigma}G_\sigma$, for $\sigma$ varying over all archimedean embeddings of $k$. Let
$G'\colonequals\mathrm{Res}_{k/\Q}(G)$. We direct the reader to  \cite[Cor. 6.1.4]{Zim}, \cite[Prop. 5.5.8]{DWM} and \cite[Prop. 5.5.10]{DWM} for a proof of the following proposition. 
\begin{proposition}
\label{ros}
	Let $G$ be a semisimple algebraic group defined over an algebraic number field
	$k$ and $G'=\mathrm{Res}_{k/\Q}(G)$. Then the following holds.\\
(1) ${G}(\mathcal{O}_k)$ embeds as an arithmetic subgroup of $G'(\R)$ via the diagonal embedding of $k$ in $\prod_{v\in\Sigma_k^\infty}k_v$.\\
(2) If $G$ is $k$-simple, then $G(\mathcal{O}_k)$ embeds as an irreducible arithmetic subgroup of $G'(\R)$.\\
(3) If $G_\sigma$ is compact for some $\sigma\in\Sigma^\infty_k$, then $G(\mathcal{O}_k)$ is cocompact.     
\end{proposition}

\subsection{Groups of type $A_1$: $\Q$-forms and Lattices}\label{lattice-in-SL2}

Simple algebraic groups defined over a number field $k$ of type $A_n$ are the groups associated to the reduced norm $1$ elements of a central simple algebra defined over $k$ (see \cite[Section 2.3.1]{P-R}). The $\bar{k}$-rational points of these groups are the special linear groups $\SL_{n+1}$. Given our interest in arithmetic subgroups of the simple Lie group $\SL_2(\R)$, we restrict $k$ to be a totally real number field unless specified otherwise and consider quaternion central simple algebras defined over $k$ \cite[Chapter 5]{SK}. The classification of arithmetic subgroups of the classical groups is described in \cite[Section 18.5]{DWM}.  

\noindent
Let $A$ be a quaternion central simple algebra defined over $k$. By Artin-Wedderburn theorem, when $\dim_k(A)=4$, there are the following two possibilities: (1) $A=D$, a quaternion division algebra, or (2) $A=\M_2(k)$, the algebra of split quaternions. In the former case, the left regular representation of $D$ enables us to construct an algebraic group defined over $k$, which is denoted $\mathbf{SL}_{1,D}$, such that $$\mathbf{SL}_{1,D}(k)=\SL_1(D)\colonequals\{b\in D^* : \rm{Nrd}(b)=1\},$$ where $\mathrm{Nrd}:D^*\rightarrow k^*$ is the {\it reduced norm} map. Using the isomorphism $D\otimes_k\bar{k}\simeq \M_2(\bar{k})$, we see that $\mathbf{SL}_{1,D}$ is a simply connected simple group of type $A_{1}$ defined over $k$. 

\noindent
\cite[Prop. 6.2.6 ]{DWM} tells that an arithmetic subgroup $\Gamma$ of $\SL_2(\R)$ is {\it cocompact (uniform)} if and only if the corresponding central simple algebra is a quaternion {\it division} algebra $D$ defined over a totally real number field $k$ and which is ramified at all but one place in $\Sigma_k^\infty$. 
In this case $\Gamma$ is commensurable to $\mathrm{SL}_1(\mathcal{O}_D)$, where $\mathcal{O}_D$ is a maximal order in $D$ and the cocompactness of $\SL_1(\mathcal O_D)$ in $\SL_2(\R)$ is proved using Godement criterion. Moreover, non-isomorphic quaternion central simple algebras define different commensurability classes of arithmetic subgroups of $\SL_2(\R)$ and all such classes arise in this way.

\noindent
When $k\neq \Q$, we appeal to Proposition \ref{ros} to obtain arithmetic subgroups of $\SL_2(\R)$ as described below. By the assumption on $D$,
\[D\otimes_{k}\R\cong\M_2(\R)\times\underbrace{\mathbb{H}\times\cdots\times\mathbb{H}}_{(m-1)\mbox{ times}},\] 
where $\mathbb{H}$ is the algebra of Hamilton's quaternions over $\R$. Using Proposition \ref{ros} we see that $\mathbf{SL}_{1,D}(\mathcal{O}_k)$ embeds as an arithmetic lattice in 
\begin{align*} 
\mathrm{Res}_{k/\Q}\mathbf{SL}_{1,D}(\R) &=\SL_1(D\otimes_{k}\R)\\ 
&=\SL_1(\M_2(\R)\times\underbrace{\mathbb{H}\times\mathbb{H}\times\cdots\times\mathbb{H}}_{(m-1)\ \mbox{times}})\\
 &=\SL_2(\R)\times\underbrace{\mathrm{SU}(2)\times\mathrm{SU}(2)\times\cdots\times\mathrm{SU}(2).}_{(m-1)\ \mbox{times}} 
 \numberthis \label{eqn1}
\end{align*}

\noindent
Thus, going modulo the compact factors of $\mathrm{SU}(2)$ in \eqref{eqn1} we obtain a cocompact arithmetic lattice in $\SL_2(\R)$. In the case when the quaternion central simple algebra is split, we get the non-cocompact arithmetic lattice $\Gamma\cong\SL_2(\Z)$ of $\SL_2(\R)$. In fact, from \cite[Prop. 6.1.5]{DWM} we see that it is the only non-cocompact arithmetic subgroup in $\SL_2(\R)$ up to commensurability and conjugates. We summarize the above discussions in the following proposition.
\begin{proposition}
For the simple group $\SL_2(\R)$, up to conjugates and commensurability, $\SL_2(\Z)$ and $\SL_1(\mathcal{O}_D)$, where $D$ is a quaternion central division algebra defined over $k$ and split at exactly one real place of $k$, are all the arithmetic subgroups. Moreover, the cocompact arithmetic subgroups are of the form $\SL_1(\mathcal{O}_D)$.

\end{proposition}


\subsection{Representation of $\SL_1(D)$ on $D$}
\noindent
We end this section with an important observation about the left multiplicative action of $\SL_1(D)$ on $D$. Observe that a quaternion central division algebra $D$ defined over $k$ necessarily splits over the algebraic closure $\bar{k}$ of $k$ so that, $D\otimes_{k}{\bar{k}}\simeq\M_2(\bar{k})$. Thus, for any embedding $\psi:D\hookrightarrow\M_2(\bar{k})$, the norm 1 elements $\SL_1(D)$ of $D$ embeds inside $\SL_2(\bar{k})$. The left regular action of $\SL_1(D)$ on $D$ is given by restricting the left matrix multiplication action of $\SL_2(\bar{k})$ on $\M_2(\bar{k})$ to the respective images of $\SL_1(D)$ and $D$. Note that the latter representation is not irreducible. Indeed, under the embedding $\psi$, $\M_2(\bar{k})\simeq{\bar{k}}^2\oplus\bar{k}^2$ as $\SL_2(\bar{k})$-representation, where $\bar{k}^2$ denotes the standard two-dimensional representation of $\SL_2(\bar{k})$.
We denote the standard $2$-dimensional representation of $\mathbf{SL}_{1,D}(\bar{k})$ on the space $\bar{k}^2$ by $\rho$. This is an algebraic representation defined over $\bar{k}$. The group $\mathbf{SL}_{1,D}(\bar{k})$ also acts algebraically on ${\rm{Sym}}^{n}(\bar{k}^2)\simeq\bar{k}^{n+1}$ for $n\in\mathbb N$ and we denote these representations by $\mathrm{Sym}^{n}\rho$. The representation $\rho$ and its symmetric powers $\mathrm{Sym}^{n}\rho$ are irreducible.

\section{Representations of $\mathbf{SL}_{1,D}(\bar{k})$}\label{s2}
\noindent
In order to prove Theorem \ref{maintheorem}, one needs to understand when is the representation $\mathrm{Sym}^m\rho$ of the algebraic $\bar{k}$-group $\mathbf{SL}_{1,D}(\bar{k})$ on ${\rm{Sym}}^{m}(\bar{k}^2)$ defined over $k$. In this section we prove a few preliminary results in this direction. We continue with the notation in the previous section so that $D$ is a quaternion central division algebra over a totally real number field $k$ which is split at exactly one real place of $k$.

\begin{lemma}\label{action on k^2}
  The standard representation $\rho$ of $\mathbf{SL}_{1,D}(\bar{k})$ on $\bar{k}^2$ is not defined over $k$.
\end{lemma}
\begin{proof}
 Suppose on the contrary that $\exists$ a $\bar{k}$-basis with respect to which $\rho(\SL_1(D))\subseteq\SL_2(k)$. As an application of Hilbert's theorem 90, it is easily seen that the regular representation of $\SL_1(D)$ on $D$ is irreducible. Thus, the algebra generated by $\SL_1(D)$, which is also an $\SL_1(D)$-invariant subspace of $D$, is all of $D$. Since $\rho$ is faithful, it follows that $D\subseteq\M_2(k)$; but, $D$ and $\M_2(k)$ have the same dimension over $k$. Therefore, $D=\M_2(k)$ as an algebra, giving a contradiction.
\end{proof}

\begin{lemma}\label{action on k^3}
The representation $\mathrm{Sym}^2\rho$ of $\mathbf{SL}_{1,D}(\bar{k})$ on ${\rm{Sym}}^2(\bar{k}^2)$ is defined over $k$.
\end{lemma}
\begin{proof}
 The proof proceeds by exhibiting a $k$-structure on $\mathrm{Sym}^2(\bar{k}^2)$ such that the representation $\mathrm{Sym}^2\rho$ with respect to this $k$-structure descends to a representation defined over $k$. 
Define the representation $\rho'$ of $\mathbf{SL}_{1,D}(\bar{k})$ on the $3$-dimensional $\bar{k}$ subspace $V$ of trace $0$ matrices in $\M_2(\bar{k})$ as 
$$\rho'(g)(v)=gvg^{-1},\ \forall{g}\in{\mathbf{SL}_{1,D}(\bar{k})}\ \mbox{and}\ \forall{v}\in{V}.$$
 Let $D^{0}$ be the subspace of trace 0 elements in $D$. As $D$ splits over $\bar{k}$, $V= D^{0}\otimes_{k}{\bar{k}}$. Then $D^{0}\hookrightarrow V$ is a $3$-dimensional $k$-subspace which  is stable under the action of $\SL_1(D)$. Thus, the representation ($V,\rho'$) is defined over $k$. By the highest weight theory for $\SL_2$ representations, $V\cong{\rm{Sym}}^2(\bar{k}^2)$. Hence, the lemma. 
\end{proof}

\noindent
We will extend the dichotomy provided by the Lemmas \ref{action on k^2} and \ref{action on k^3} to higher dimensional representations of $\mathbf{SL}_{1,D}$ in Proposition \ref{def|k}. The next proposition applies to reducible representations of $\mathbf{SL}_{1,D}(\bar{k})$ on $\bar{k}^{m+1}$. The proposition holds true in greater generality for any semisimple representation of an abstract group on a $k$-vector space $V$ and isotypic $G$-submodules of $V\otimes_k\bar{k}$. We prove the result here for a reducible representation of a semisimple algebraic group $G$ defined over $k$ and $G(k)$-submodules of $V\otimes_k\bar{k}$ having multiplicity one. Appendix B comprises the proof of the general statement.

\begin{proposition}\label{direct summand def over k}

Let $V$ be an $n$-dimensional vector space defined over $k$ and $G$ be a semisimple algebraic group defined over $k$. Let $\phi$ be an algebraic representation of $G$ on $V$ defined over $k$. Suppose $V_i\subset V\otimes_k\bar{k},$ for $i=1,\ldots,l$ are irreducible $G(k)$-invariant subspaces defined over $\bar{k}$, each appearing with multiplicity one in $V\otimes_k\bar{k}$. Then each $(V_{i},\phi_{i})$ is defined over $k$, where $\phi_{i}$ is $\phi\mid_{V_{i}}$.
\end{proposition}

\begin{proof} 
It suffices to find a $k$-subspace $W_{i}$ of $V_{i}$ which is $G(k)$-stable and such that $V_{i}\cong W_{i}\otimes_{k}\bar{k},\mbox{ for } 1\leq{i}\leq{l}$. 
For any $\sigma\neq\mathrm{id}$ in $\mathrm{Gal}(\bar{k}/k)$, the map  $f_{\sigma}\colonequals\mbox{I}_{n}\otimes_{k}\sigma$, where $\mbox{I}_{n}$ is the identity map on $V$, is a $k$-linear isomorphism of $V\otimes_{k}\bar{k}$ onto itself which commutes with $\phi{(g)}$ for all $g\in G(k)$.
Hence, since $V_i$ are irreducible $G(k)$-submodules of $V\otimes\bar{k}$, so are $f_{\sigma}(V_{i})$ for $i\in\{1,\ldots,l\}$. Thus, $$f_{\sigma}(V_{i})=V_{i}, \ \forall\  1\leq{i}\leq{l}.$$
To show the existence of the $k$-subspace $W_i$, we need to see that $V\cap{V_{i}}\neq\{0\}$. Suppose $v$ be any non zero element in $V_i$. Let $E$ be a finite extension of $k$ generated by the coordinates of $v$ with respect to a $k$-basis $\mathcal{B}$ of $V$. If necessary, modify $v$ to $\tilde{v}$ to get at least one non-zero coordinate of each $v$ with respect to $\mathcal B$ as $1$. Then  the coordinates of $\omega\colonequals\sum_{\sigma\in\mathrm{Gal}(E/k)}f_\sigma(v)$ are the trace from $E$ to $k$ of the coordinates of $v$. Since one of the coordinate is assumed to be $1$, $\omega$ is a non zero element lying in $V\cap V_i$.
Thus, $W_i\colonequals V_i\cap{V}$ is a non trivial subspace of $V_i$ defined over $k$ and it is easily seen to be a $G(k)$-submodule of $V_{i}$. Since each $V_i$ is irreducible, we have $V_i\cong W_i\otimes_k\bar{k}$.
\end{proof}

\begin{proposition}\label{def|k}
The irreducible representation $\mathrm{Sym}^m\rho$ of $\mathbf{SL}_{1,D}(\bar{k})$ on ${\mathrm{Sym}}^{m}(\bar{k}^2)=\bar{k}^{m+1}$ is defined over ${k}$ if and only if $m$ is even.
\end{proposition}
\begin{proof}
By the highest weight theory, for every $\ell\in\Z_{\geq 0}$, we have up to isomorphism a unique irreducible representation $\mathrm{Sym}^{\ell}\rho$ of $\SL_2$ with highest weight $\ell$. If the representations $\mathrm{Sym}^r\rho$ and $\mathrm{Sym}^s\rho$ are both defined over $k$, then it is immediately seen that the tensor product representation ${\mathrm{Sym}}^{r}\rho\otimes_{\bar{k}}{\mathrm{Sym}}^{s}\rho$ is defined over $k$. The complete reducibility of this tensor product representation of $\mathbf{SL}_{1,D}(\bar{k})$ is explicitly given by the Clebsch-Gordan formula \cite[Prop. 5.5]{B-D}:
\begin{equation}\label{CGFormula}
  {\rm{Sym}}^{r}(\bar{k}^2)\otimes_{\bar{k}}{\rm{Sym}}^{s}(\bar{k}^2)=\bigoplus_{i=0}^{\min\{r,s\}}{\rm{Sym}}^{r+s-2i}(\bar{k}^2).  
\end{equation}
\noindent
Note that since $\dim_{\bar{k}}{\rm{Sym}}^{r+s-2i}(\bar{k}^2)=(r+s-2i)+1$, each irreducible subrepresentations has multiplicity one. 
Then, by Proposition \ref{direct summand def over k}, the representation of $\mathbf{SL}_{1,D}(\bar{k})$ on ${\rm{Sym}}^{r}(\bar{k}^2)\otimes_{\bar{k}}{\rm{Sym}}^{s}(\bar{k}^2)$ is defined over $k$ if and only if each irreducible subrepresentation $\mathrm{Sym}^{r+s-2i}\rho$, for $i\in\{0,1,\ldots,\mbox{min}\{r,s\}\}$, is defined over $k$.

\noindent
We first prove that when $m=2n$ is even, $\mathrm{Sym}^m\rho$ is defined over $k$. The strategy is to use induction on $n$. The trivial representation is vacuously defined over $k$. When $n=1$, Lemma \ref{action on k^3} upholds the result. Let $n > 1$. Assume that the representations $\mathrm{Sym}^{2i}\rho$ are all defined over $k$ for $i$ varying over the nonnegative integers up to $n-1$.  Choose nonnegative even integers $r,s$ such that $0< r,s<m$ and $r+s>m=2n$. Observe that $\mathrm{Sym}^m(\bar{k}^2)$ appears as a direct summand in \eqref{CGFormula} when $i=(r+s-m)/2$. By induction hypothesis, $\mathrm{Sym}^j(\bar{k}^2)$ is defined over $k$ for $j=r,s$ so that the tensor product on the left hand side is defined over $k$. Now, Proposition \ref{direct summand def over k} forces each irreducible representation on the right hand side of \eqref{CGFormula} to be defined over $k$. Hence, $\mathrm{Sym}^{m}\rho$ is defined over $k$. 

\noindent
To prove the only if part, it is enough to show that when $m=2n+1$ is an odd integer the representation $\mathrm{Sym}^m\rho$ is not defined over $k$. On the contrary, if possible, suppose that $\mathrm{Sym}^{m}\rho$ is defined over $k$, so that the tensor product representation $\mathrm{Sym}^{m}\rho\otimes_{\bar{k}}\mathrm{Sym}^2\rho$ is defined over $k$. The decomposition in \eqref{CGFormula} together with Proposition \ref{direct summand def over k} then implies that the representations $\mathrm{Sym}^{2n+3}\rho$ and $\mathrm{Sym}^{2n-1}\rho$ are both defined over k. Reiterating the above procedure, we eventually get that $\mathrm{Sym}^1\rho=\rho$ is defined over $k$.
This contradicts Lemma \ref{action on k^2}. 
Thus, $\mathrm{Sym}^{m}(\bar{k}^2)$ for $m$ odd, is not defined over $k$. Hence the proposition.
\end{proof}

\begin{rmk}
    We note here that the Clebsh-Gordon formula for decomposition of tensor product of representation is known only for $\SL_2$ and hence this limits the extension of the present methods to higher rank.

\end{rmk}

\section{Lattices in $\R^n\rtimes\SL_2(\R)$}\label{s3}

\noindent
In this section we detail on the discussions surrounding Theorem \ref{maintheorem} stated in the Introduction. Notation from the previous sections are used here without further definition. Let $\Gamma$ be a lattice in $\R^n\rtimes_\varphi\SL_2(\R)$, where $\varphi$ is not necessarily an irreducible representation of $\SL_2(\R)$ on $\R^n$. Let $\Gamma_1=\Gamma\cap\R^n$ and $\Gamma_2$ be the projection of $\Gamma$ on $\SL_2(\R)$. By \cite[Prop. 1.3]{TSW} both $\Gamma_1$ and $\Gamma_2$ are lattices. It is interesting to see that the very existence of the lattice $\Gamma$ necessitates $\Gamma_2$ to be an arithmetic subgroup of $\SL_2(\R)$.

\begin{proposition}\label{arithmetic-1}
 For $n\geq{2}$, let $\varphi$ be any nontrivial representation of $\SL_2(\R)$ on $\R^n$ and let $\Gamma\subset\R^n\rtimes_\varphi\SL_2(\R)$ be a lattice. Then, the lattice $\Gamma_2$ in $\SL_2(\R)$ is an arithmetic lattice.
\end{proposition}

\begin{proof}
  Suppose the representation $\varphi:\SL_2(\R)\rightarrow\SL_{n}(\R)$ is faithful. Since $\Gamma_1$ is a normal subgroup in $\Gamma$, $\Gamma_2$ is contained in $\Lambda\colonequals\{g\in\SL_2(\R): \varphi(g)\Gamma_1=\Gamma_1\}$ which is easily seen to be an arithmetic subgroup of $\SL_2(\R)$. Indeed, $\varphi(\Lambda)=\varphi(\SL_2(\R))\cap{a\SL_{n}(\Z)a^{-1}}$ for some $a\in\GL_{n}(\R)$ which conjugates $\Gamma_1$ to the lattice $\Z^n$ given by the standard $\Q$-structure in $\R^{n}$. Since, $\varphi(\Gamma_2)$ is a lattice in $\varphi(\Lambda)$, this implies that the Zariski closure of $\varphi(\Lambda)$ determines an algebraic subgroup $G$ defined over $\Q$ of $\mathbf{SL}_n$ such that, $\varphi(\Gamma_2)\subset G(\Q)$ and $G(\R)=\SL_2(\R)$. It follows that $\Gamma_2$ is an  arithmetic subgroup of $\SL_2(\R)$.

\noindent
 Now, given any nontrivial representation $\varphi$, either $\varphi$ is faithful or $\mathrm{Ker}(\varphi)=\{\pm{\mathrm{I}}\}$ which is a compact subgroup of $\SL_2(\R)$. The latter case gives a faithful representation of $\mathrm{PSL}_2(\R)$ on $\R^{n}$. In the light of the arguments in the preceding paragraph, we obtain an algebraic group $G$ defined over $\Q$ such that $G(\R) = \mathrm{PSL}_2(\R)$ and the image of $\Gamma_2$ in $\mathrm{PSL}_2(\R)$ is arithmetic subgroup with respect to this $\Q$-structure. Consider the simply connected cover $\widetilde{G}$ of $G$. Since $G$ is a $\Q$-form of $\mathbf{PSL}_2$, we know that its (algebraically) simply connected cover will be a $\Q$-form of the group $\mathbf{SL}_2$. Since $\widetilde{G}\rightarrow G$ is an isogeny we get that $\Gamma_2$ is an arithmetic subgroup of $\SL_2(\R)$ with respect to the $\Q$-structure given by $\widetilde{G}$.
\end{proof}

\begin{rmk}\label{existenceofQ-forminSL_2}
    A careful reading of the proof of Proposition \ref{arithmetic-1} reveals that regardless of the parity of $n$, if $\Gamma\subset\R^n\rtimes_\varphi\SL_2(\R)$ is a lattice, we get a $\Q$-form of $\SL_2(\R)$ determined by an algebraic subgroup $G$ of $\mathbf{SL}_n$ defined over $\Q$, which contains the image of $\Gamma_2$ in $\SL_n(\R)$. We emphasize here that the existence of such a $\Q$-form of $\SL_2(\R)$ depends only on the fact that $\Gamma_2$ stabilizes a lattice $\Gamma_1$ in $\R^n$ and is irrespective of the action being irreducible or not. 
\end{rmk}
\begin{proposition}\label{D_definedover_Q}
    If a lattice $\Gamma_2$ in $\SL_2(\R)$ stabilizes a lattice $\Gamma_1$ in $\R^n$ for some $n$, then $\Gamma_2$ is an arithmetic lattice in $\SL_2(\R)$ determined by a quaternion central simple algebra defined over $\Q$. 
\end{proposition}
\begin{proof}
   The proof of Proposition \ref{arithmetic-1}, and Remark \ref{existenceofQ-forminSL_2} yields that $\Gamma_2$ is an arithmetic lattice with respect to a $\Q$-form $G$ of $\SL_2(\R)$. Since the $\Q$-forms of $\SL_2(\R)$ are precisely $\mathbf{SL}_2$ and $\mathbf{SL}_{1,D}$ for quaternion division algebras $D$ defined over $\Q$, we conclude that $G=\mathbf{SL}_{1,D}$ or $\mathbf{SL}_2$. 
\end{proof}
\begin{proposition}\label{ccpt=def|Q}
  If $\Gamma\subset \R^n\rtimes_\varphi\SL_2(\R)$ is a cocompact lattice, then the nontrivial representation $\varphi$ of $\SL_2(\R)$ on $\R^n$ is defined over $\Q$.
\end{proposition}
\begin{proof}

\noindent
 Let $\Gamma_2$ be the projection of $\Gamma$ on $\SL_2(\R)$. Since $\Gamma_2$ is a cocompact arithmetic lattice in $\SL_2(\R)$, by Proposition \ref{D_definedover_Q} $\Gamma_2=\SL_1(\mathcal{O}_D)$ for some quaternion central division algebra $D$ defined over $\Q$. Being a lattice, $\Gamma_2$ is Zariski dense in $\SL_2(\R)$ by Borel density theorem. If $F=\Q(c_1,\ldots,c_r)$ is an extension of $\Q$ by the coefficients of the polynomials $f$ defining $\varphi$, then for any $\sigma\in\mathrm{Gal}(F/\Q)$, $\Gamma_2$ intersects the Zariski open set $U_{\sigma}\colonequals\{p\in\SL_2(\R):\sigma{(f)}(p)\neq{f(p)}\}$ nontrivially if $U_\sigma\neq\emptyset$. This gives a contradiction as $f(p)\in\Z$ for all $p\in\Gamma_2$. Thus, $U_\sigma=\emptyset$ so that $F=\Q$. 
\end{proof}

\begin{lemma}\label{def|Q=ccpt}
    For a quaternion division algebra $D$ defined over $\Q$, if a representation $\varphi:\mathbf{SL}_{1,D}(\R)=\SL_2(\R)\rightarrow\SL_n(\R)$ is defined over $\Q$, then $\R^n\rtimes_\varphi\SL_2(\R)$ has a cocompact lattice.
\end{lemma}
\begin{proof}
Lemma 4.1 of \cite{P-R} shows the existence of a finite index subgroup $H$ of $\SL_1(\mathcal{O}_D)$ when $\varphi$ is defined over $\Q$, such that $\varphi(\mathrm{H})\subset\SL_n(\Z)$ with respect to an $\R$-basis $\{v_1,\ldots,v_n\}$ of $\R^n$. That is, $H$ stabilizes an arithmetic lattice $\Gamma_1$ in $\R^n$. Hence, there exists a cocompact lattice in $\R^n\rtimes_\varphi{\SL_2(\R)}$ projecting onto $\SL_1(\mathcal{O}_D)$ in $\SL_2(\Z)$. 
\end{proof} 
\noindent
At this point we need to be mindful about the parity of $n$ when the action of $\SL_2(\R)$ on $\R^n$ is irreducible. The necessary condition stated in the Proposition \ref{ccpt=def|Q} does not hold true when $n$ is an even integer. In Lemma \ref{on R^2n not def|Q}, $D$ continues to be a quaternion division algebra defined over $\Q$ and split over $\R$. We let $G=\mathbf{SL}_{1,D}$ so that $G(\R)=\SL_2(\R)$.
\begin{lemma}\label{on R^2n not def|Q}
    For $n\geq{2}$, the irreducible representation of $G(\R)$ on $\R^n$ is defined over $\Q$ if and only if $n$ is odd.
\end{lemma}
\begin{proof}
   This result follows from Proposition \ref{def|k} if we restrict the representation defined over the completion $\R$ to $\Q$ and see that the induced irreducible representation of $G(\bar{\Q})$ on $\mathrm{Sym}^{n-1}(\bar{\Q}^2)$, which is a priori defined over the algebraic closure $\bar{\Q}$, is defined over $\Q$ if and only if $n$ is odd.
   
\end{proof}

\section{Proof of main theorems}\label{s4}
\noindent
The goal of this section is to prove Theorem \ref{maintheorem} and Theorem \ref{secondmaintheorem} stated in the Introduction. With the detailed analysis on $\Q$-structures done in the preceding section, we need only combine the relevant results to prove Theorem \ref{maintheorem}. 

\subsection{Proof of Theorem \ref{maintheorem}}

\begin{theorem}[Part (1), Theorem \ref{maintheorem}]
\label{ccpt=n odd}
   For $n\geq{2}$, and the irreducible action of $\SL_2(\R)$ on $\R^n \ (\simeq{\rm{Sym}}^{n-1}(\R^2))$, the group $\R^n\rtimes\SL_2(\R)$ contains a cocompact lattice if and only if $n$ is odd. Moreover, the cocompact lattices are of the form $\Z^n\rtimes\SL_1(\mathcal O_D)$ upto conjugates and commensurability, where $D$ is a quaternion division algebra over a totally real number field $k$.
\end{theorem}
\begin{proof}
A $\Q$-structure on $\R^n$ invariant under an irreducible action of a $\Q$-structure $\mathbf{SL}_{1,D}$ on $\SL_2(\R)$ exists only for an odd dimensional representation is seen from Lemma \ref{on R^2n not def|Q}. The result now follows from Lemma \ref{def|Q=ccpt}. 

When $\Gamma$ is a cocompact lattice in $\R^n\rtimes\SL_2(\R)$, the projection $\Gamma_2$ of $\Gamma$ onto $\SL_2(\R)$ is a cocompact arithmetic lattice in $\SL_2(\R)$ and has the form $\SL_1(\mathcal O_D)$ for a quaternion division algebra $D$ over a totally real number field $k$. Then the structure of $\Gamma$ follows from the arguments in Theorem \ref{lattices in R2.SL(2,R)} below.
\end{proof}

\noindent
Note that for any $n\in\N\cup\{0\}$, $\R^n\rtimes\SL_2(\R)$ contains a {\it non-cocompact} lattice of the form $\Z^n\rtimes\SL_2(\Z)$. In the following theorem we prove that for an irreducible representation of $\SL_2(\R)$ of even dimension $n$, all lattices in $\R^n\rtimes\SL_2(\R)$ are of the form $\Z^n\rtimes\SL_2(\Z)$. This will complete the proof of Theorem \ref{maintheorem}.

\begin{theorem}[Part (2), Theorem \ref{maintheorem}]
\label{lattices in R2.SL(2,R)}
    For a positive even integer $n$ and the irreducible action of $\SL_2(\R)$ on $\R^n$, every lattice in $\R^n\rtimes\SL_2(\R)$ is isomorphic to $\Z^n\rtimes\SL_2(\Z)$ up to conjugates and commensurability.
\end{theorem}

\begin{proof}
    Let $\pi:\R^n\rtimes\SL_2(\R)\rightarrow\SL_2(\R)$ denote the projection map. Then, for a lattice $\Gamma\subset\R^n\rtimes\SL_2(\R)$, $\Gamma_1=\Gamma\cap{\R^{n}}$ and $\Gamma_2= \pi(\Gamma)$ are lattices in $\R^{n}$ and $\SL_2(\R)$, respectively. We have the exact sequence of lattices
    $$0\longrightarrow\Gamma_1\longrightarrow\Gamma\overset{\pi}\longrightarrow\Gamma_2\longrightarrow{1}.$$

    \noindent
    When $n$ is even, by Lemma \ref{on R^2n not def|Q} and the notation therein, the representation of $\mathbf{SL}_{1,D}(\R)$ on $\R^n$ is not defined over $\Q$. By Theorem \ref{ccpt=n odd} this implies that $\Gamma_2$ is a {\it non-cocompact} arithmetic lattice in $\SL_2(\R)$ and hence, is the group of $\Z$-points of the split central simple algebra $\M_2(\Q)$ over $\Q$ (see Section \ref{notation}). Therefore, $\Gamma_2$ is commensurable with $\SL_2(\Z)$, and thus, contains a free subgroup $\Lambda$ having a finite index in $\SL_2(\Z)$. Indeed, for any free subgroup $\Tilde{\Lambda}$ of finite index in $\SL_2(\Z)$, $\Lambda=\Gamma_2\cap\Tilde\Lambda$ is a free subgroup of finite index in $\Gamma_2\cap\SL_2(\Z)$.
    The inverse image $\pi^{-1}(\Lambda)$ of $\Lambda$, contained inside $\Gamma$, gives the exact sequence 
    \begin{equation}\label{inverseImgLambda}
    0\longrightarrow\Gamma_1\longrightarrow{\pi^{-1}(\Lambda)}\overset{\pi}\longrightarrow\Lambda\longrightarrow{1}.
    \end{equation}
     Now, it can easily be seen that the map $\pi^*:\Gamma/\pi^{-1}(\Lambda)\longrightarrow\Gamma_2/\Lambda$, defined as $\pi^*([\gamma])\colonequals[\pi(\gamma)]$ for all $\gamma\in\Gamma$, is a bijection. Therefore $\pi^{-1}(\Lambda)$ is a subgroup of finite index in $\Gamma$ and hence, a lattice in $\R^{n}\rtimes\SL_2(\R)$.
   
    \noindent
    Since $\Lambda$ is a free group, the sequence \eqref{inverseImgLambda} splits, giving $\pi^{-1}(\Lambda)\cong\Gamma_1\rtimes\Lambda$. Up to conjugates and commensurability, every lattice in $\R^{n}$ is of the form $\Z^n$. Thus, $\Gamma\cong\Z^{n}\rtimes\SL_2(\Z)$ up to conjugation and commensurability.
\end{proof}

\subsection{Proof of Theorem \ref{secondmaintheorem}}
We now proceed to consider a non-trivial and not necessarily irreducible representation $\varphi$ of $\SL_2(\R)$ on $\R^n$. We continue with the notation from the previous sections. If the lattice $\Gamma\subset \R^n\rtimes_\varphi\SL_2(\R)$ is cocompact, then the lattice $\Gamma_2$ is cocompact in $\SL_2(\R)$. By Proposition \ref{D_definedover_Q}, $\Gamma_2=\SL_1(\mathcal O_D)$ for a quaternion division algebra $D$ defined over $\Q$ and split over $\R$. Then by Proposition \ref{ccpt=def|Q}, the representation of $\mathbf{SL}_{1,D}(\R)$ on $\R^n$ is defined over $\Q$. A converse to this statement is true with a certain restriction on the multiplicities of the irreducible subrepresentations of even dimension occurring in $\R^n$.
\begin{proposition}\label{for reducible action def|Q=even multiplicity}
A representation of $\mathbf{SL}_{1,D}(\R)$ on $\R^n$ is defined over $\Q$ if and only if each irreducible subrepresentation of even dimension in $\R^n$ has even multiplicity.
\end{proposition}
\begin{proof}
To prove the sufficiency assertion, we need only show that for any even integer $m$, the representation of $\mathbf{SL}_{1,D}(\R)$ on $\mathrm{Sym}^{m-1}(\R^2)\oplus\mathrm{Sym}^{m-1}(\R^2)$, is defined over $\Q$.  By Proposition \ref{ccpt=def|Q} the representation of $\mathbf{SL}_{1,D}(\R)$ on $D\otimes_{\Q}\R\cong\M_2(\R)$ is defined over $\Q$. The discussion towards the end of Section \ref{lattice-in-SL2} shows that $\M_2(\R)$ decomposes into $\R^2\oplus\R^2$ as an $\mathbf{SL}_{1,D}(\R)$-module. Thus, the representation of $\mathbf{SL}_{1,D}(\R)$ on $\R^2\oplus\R^2\cong\mathrm{Sym}^{1}(\R^2)\oplus\mathrm{Sym}^{1}(\R^2)$ is defined over $\Q$. 

\noindent
By Lemma \ref{on R^2n not def|Q} the representation of $\mathbf{SL}_{1,D}(\R)$ on $\R^3\cong\rm{Sym}^{2}(\R^2)$ is defined over $\Q$. Following the isomorphism of $\mathbf{SL}_{1,D}(\R)$-modules using the Clebsch-Gordon formula in \eqref{CGFormula},
$$\underbrace{(\rm{Sym}^{1}(\R^2)\oplus\rm{Sym}^{1}(\R^2))\otimes_{\Q}\rm{Sym}^{2}(\R^2)}_{\mbox{defined over }\Q}\cong\underbrace{\rm{Sym}^{3}(\R^2)\oplus\rm{Sym}^{3}(\R^2)}\oplus\underbrace{\rm{Sym}^{1}(\R^2)\oplus\rm{Sym}^{1}(\R^2)}_{\mbox{defined over }\Q},$$ 
the representation of $\mathbf{SL}_{1,D}(\R)$ on $\mathrm{Sym}^{3}(\R^2)\oplus\mathrm{Sym}^{3}(\R^2)\cong\R^4\oplus\R^4$ is defined over $\Q$.
Now, use induction on the dimension $m$ to see that the representation of $\mathbf{SL}_{1,D}(\R)$ on $\R^m\oplus\R^m\cong\rm{Sym}^{m-1}(\R^2)\oplus\rm{Sym}^{m-1}(\R^2)$ is defined over $\Q$ when $m$ is an even integer. 

\noindent
For the necessary assertion, assume that the representation of $\mathbf{SL}_{1,D}(\R)$ on $\R^n$ be defined over $\Q$, and let 
$$\R^n\cong\underbrace{(\oplus_{i=1}^{s_1}(\mathrm{Sym}^{2l_i-1}(\R^2)\oplus\mathrm{Sym}^{2l_i-1}(\R^2)))}_{\mbox{defined over }\Q}\oplus\underbrace{(\oplus_{i=1}^{s_2}\mathrm{Sym}^{2m_i}(\R^2))}_{\mbox{defined over }\Q}\oplus\ (\oplus_{i=1}^{s_3}\mathrm{Sym}^{2n_i-1}(\R^2))$$ be the decomposition of $\R^n$ into irreducible $\mathbf{SL}_{1,D}(\R)$-submodules, where $n_i\neq{n_j}$ for $i\neq{j}$. From the arguments in the `if' part above and Lemma \ref{on R^2n not def|Q}, the terms in the first two parenthesis are defined over $\Q$. Assuming the $\mathbf{SL}_{1,D}(\R)$-submodule $\oplus_{i=1}^{s_3}\mathrm{Sym}^{2n_i-1}(\R^2)$ is non-zero, we get that the representation of $\mathbf{SL}_{1,D}(\R)$ on $\oplus_{i=1}^{s_3}\mathrm{Sym}^{2n_i-1}(\R^2)$ is defined over $\Q$. Hence, by Proposition \ref{direct summand def over k}, we see that the representation of $\mathbf{SL}_{1,D}(\R)$ on $\mathrm{Sym}^{2n_i-1}(\R^2)$, for each $i$, is defined over $\Q$. This is a contradiction to Lemma \ref{on R^2n not def|Q}. Thus, $\oplus_{i=1}^{s_3}\mathrm{Sym}^{2n_i-1}(\R^2)$ is trivial. That is, each irreducible representation of even dimension which occurs in $\R^n$ has even multiplicity.
\end{proof}

\begin{proposition}\label{for reducible action def|Q=ccpt}
If the representation of $\mathbf{SL}_{1,D}(\R)$ on $\R^n$ is defined over $\Q$, then  $\R^n\rtimes_\varphi\SL_2(\R)$ contains a cocompact lattice.
\end{proposition}
\begin{proof}
Proof of Lemma \ref{def|Q=ccpt} does not depend on the irreducibility of the representation of $\mathbf{SL}_{1,D}(\R)$ on $\R^n$.
\end{proof}
\begin{proof}[{\bf Proof of Theorem \ref{secondmaintheorem}:}] 
Proposition \ref{ccpt=def|Q} states that, if $\R^n\rtimes_\varphi\SL_2(\R)$ contains a cocompact lattice, then $\varphi$ is defined over $\Q$. The converse assertion follows from Propositions  \ref{for reducible action def|Q=even multiplicity} and \ref{for reducible action def|Q=ccpt}.
\end{proof}

\noindent
Although the cocompact lattices are of primary interest in this article, we observe that for any $n\in\N\cup\{0\}$, and any representation $\varphi$ of $\SL_2(\R)$ on $\R^n$, $\R^n\rtimes_\varphi\SL_2(\R)\cong(\oplus_{i=1}^{r}\R^{k_i})\rtimes\SL_2(\R)$ contains a non-cocompact lattice of the form $(\oplus_{i=1}^{r}\Z^{k_i})\rtimes\SL_2(\Z)$.

\appendix
\section{Comments on Corollary 8.28 in \cite{Rag-1}}
\noindent
Readers familiar with the results in \cite[Chapter 8]{Rag-1} may understand that the Corollary 8.25 therein is false, with a counter example to the same provided in \cite{TSW}, and this invalidates the proof of Corollary 8.28. However, in \cite[Prop. 1.3]{TSW}, a proof independent of Corollary 8.25 is provided and a correction in the statement of Corollary 8.28 is effected by assuming the maximal connected closed nilpotent subgroup $N$ to be normal in $G$. We record the corrected result here for reference.
\begin{theorem}[Proposition 1.3, \cite{TSW}]
    Let $G$ be a connected Lie group, $\Gamma\subset G$ be a lattice. Let $R$ be the radical of $G$ and $N$ the maximum connected closed nilpotent normal subgroup of $G$. Let $S$ be a semisimple subgroup of $G$ such that $G=S\cdot R$. Let $\sigma$ denote the action of $S$ on $R$. Assume that the kernel of $\sigma$ has no compact factor in its identity component. Let $\pi: G\rightarrow G/R$ and $\pi':G\rightarrow G/N$ be the natural maps. Then $R/(R\cap \Gamma)$ and $N/(N\cap\Gamma)$ are both compact. Moreover $\pi(\Gamma)$ and $\pi'(\Gamma)$ are lattices in $G/R$ and $G/N$ respectively.
\end{theorem}

\section{Generalization of proposition \ref{direct summand def over k}}

\noindent
Let $\eta:G\rightarrow \GL(V)$ be an $n$-dimensional semisimple representation of a group $G$ on a $k$-vector space $V$, for a number field $k$. Let $W\subset V\otimes_k\bar{k}$ be an isotypic component of the representation $\eta_{\bar{k}}:G\rightarrow\GL(V\otimes_k\bar{k})$. Then $W$ is defined over $k$.

\begin{proof}
Since $\eta_{\bar{k}}$ is semisimple, $W$ has an orthogonal complement $W'\subset V\otimes_k\bar{k}$. For any $\sigma\in\mathrm{Gal}(\bar{k}/k)$, $ I_n\otimes\sigma$ is an isomorphism of $V\otimes_k\bar{k}$ on itself commuting with $\eta(g)$ for $g\in G$. Thus, we get the following commutative diagram for any $\sigma\in\mathrm{Gal}(\bar{k}/k)$.
\[
\begin{tikzcd}
 W \arrow[hookrightarrow]{r} \arrow[d, "\eta(g)"'] & V\otimes_k\bar{k} \arrow[r] \arrow[d, "\eta(g)"'] & V\otimes_k\bar{k} \arrow[r] \arrow[d, "\eta(g)"'] & W'\arrow[d, "\eta(g)"']\\
 W \arrow[hookrightarrow]{r} & V\otimes_k\bar{k} \arrow[r] & V\otimes_k\bar{k} \arrow[r] & W'.
\end{tikzcd}
\]
Because there cannot be any nonzero maps between $W$ and $W'$ commuting with all $g\in G$, we conclude that the subspaces $W$ and $W'$ are stable under $\mathrm{Gal}(\bar{k}/k)$. \cite[Theorem 2.14]{KC} now gives the existence of the $k$-structures $W_k$ and $W'_k$ respectively of $W$ and $W'$. Thus, $W_k\oplus W'_k$ is a $k$-structure in $V\otimes_k\bar{k}$. But any vector space admits a unique $k$-structure, hence, $V=W_k\oplus W'_k$. Since $\eta$ is defined over $k$, we see that $W$ is defined over $k$.
\end{proof}

\section*{Acknowledgements}
\noindent
We are grateful to T. N. Venkataramana for suggesting this problem and for many useful discussions. We thank Dipendra Prasad for his interest in this project and suggestions. The second-named author thanks SERB for supporting the project via the MATRICS grant (MTR/2021/000368).  The first-named author is supported by a postdoctoral fellowship from the NBHM grant 02018/2/2023-R\&D-II/7839 under Kerala School of Mathematics. We are greatly indebted to the anonymous referees for their invaluable comments to improve the article.

\section*{Conflict of interest}
The authors have no conflicts of interest to declare that are relevant to this article.

\end{document}